\documentclass[reqno,twoside]{amsart}
\usepackage{eurosym}
\usepackage{amssymb}
\usepackage{amsmath}
\usepackage{amsfonts}
\usepackage{graphicx,color}
\usepackage{fancyhdr}

\newtheorem{theorem}{Theorem}
\theoremstyle{plain}

\newtheorem{lemma}{Lemma}

\newtheorem{remark}{Remark}

\numberwithin{equation}{section}
\fancypagestyle{mypagestyle}{	\fancyhf{}	\fancyhead[OC]{ \sc  Figueiredo, Moussaoui, dos Santos and Tavares}	\fancyhead[EC]{\sc Nonlocal problems involving Orlicz spaces}		\fancyhead[OR]{\thepage}
	\fancyhead[EL]{\thepage}
}
\begin{document}
\title[Nonlocal problems involving Orlicz spaces]{A sub-supersolution
approach for some classes of nonlocal problems involving Orlicz spaces}
\author{}
\maketitle
\vspace{-0.5cm}

\begin{center}
	\author{\sc Giovany J. M. Figueiredo}
\end{center}
\vspace{-0.2cm}
\begin{center}
	Universidade de Bras\'{i}lia, Departamento de Matem\'{a}tica, CEP: 70910-900, \\ Bras\'{i}lia-DF, Brazil
\end{center}
\vspace{0.1cm}
\begin{center}
	\author{\sc Abdelkrim Moussaoui}
\end{center}
\vspace{-0.2cm}
\begin{center}
	A. Mira Bejaia University, Biology Department, Targa Ouzemour, 06000, \\ Bejaia- Algeria
\end{center}
\vspace{0.1cm}
\begin{center}
	\author{\sc Gelson C.G. dos Santos}
\end{center}
\vspace{-0.2cm}
\begin{center}
	Universidade Federal do Par\'{a}, Faculdade de Matem\'{a}tica, CEP: 66075-110, Bel\'{e}m-PA, Brazil
\end{center}
\vspace{0.1cm}
\begin{center}
	\author{\sc Leandro S. Tavares}
\end{center}
\vspace{-0.2cm}
\begin{center}
	Universidade Federal do Cariri, Centro de Ci\^{e}ncias e Tecnologia, CEP:63048-080, Juazeiro do Norte-CE, Brazil
\end{center}

\begin{abstract}
In the present paper we study the existence of solutions for some nonlocal
problems involving Orlicz-Sobolev spaces. The approach is based on
sub-supersolutions.
\end{abstract}












\section{Introduction}

Let $\Omega $ be a bounded domain in $\mathbb{R}^{N}$ $(N\geq 3)$ with $C^{2}
$ boundary $\partial \Omega $. In the present paper we focus on the problems
of quasilinear elliptic nonlocal equations 
\begin{equation*}
\left\{ 
\begin{array}{rcl}
\label{problema-(P)}-\Delta _{\Phi }u & = & f(u)|u|_{{L^{\Psi }}}^{\alpha
}+g(u)|u|_{L^{\Lambda }}^{\gamma }\;\;\mbox{in}\;\;\Omega , \\ 
u & = & 0\;\;\mbox{on}\;\;\partial \Omega 
\end{array}%
\right. \eqno{(P_1)}
\end{equation*}%
and%
\begin{equation*}
\left\{ 
\begin{array}{rcl}
\label{problema-(P)}-\Delta _{\Phi _{1}}u & = & f_{1}(v)|v|_{L^{\Psi
_{1}}}^{\alpha _{1}}+g_{1}(v)|v|_{L^{\Lambda _{1}}}^{\gamma _{1}}\;\;%
\mbox{in}\;\;\Omega , \\ 
-\Delta _{\Phi _{2}}v & = & f_{2}(u)|u|_{L^{\Psi _{2}}}^{\alpha
_{2}}+g_{2}(u)|u|_{L^{\Lambda _{2}}}^{\gamma _{2}}\;\;\mbox{in}\;\;\Omega ,
\\ 
u=v & = & 0\;\;\mbox{on}\;\;\partial \Omega ,%
\end{array}%
\right. \eqno{(P_2)}
\end{equation*}%
where $\alpha _{i},\gamma _{i},$ $i=0,1,2,$ with $\alpha _{0}:=\alpha $ and $%
\gamma _{0}:=\gamma ,$ are positive constants, $|.|_{L^{\Psi }}$ (resp.\ $%
|.|_{L^{\Lambda }}$) denotes the norm in the Orlicz space $L^{\Psi }(\Omega )
$ (resp. $L^{\Lambda }(\Omega )$) and the nonlinearities $%
f_{i},g_{i}:[0,+\infty )\rightarrow \lbrack 0,+\infty ),$ $i=0,1,2,$ with $%
f_{0}:=f$ and $g_{0}:=g$, are continuous and nondecreasing functions. Here, $%
\Delta _{\Phi _{i}}$ stands for the $\Phi _{i}-$Laplacian operator, that is, 
$\Delta _{\Phi _{i}}w=\mathrm{div}\,(\phi _{i}(|\nabla w|)\nabla w),$ for $%
i=0,1,2$, where $\Phi _{i}:\mathbb{R}\rightarrow \mathbb{R}$ are $N$%
-functions of the form 
\begin{equation}
\Phi _{i}(t):=\int_{0}^{|t|}\phi _{i}(s)sds,  \label{phi}
\end{equation}%
with $\phi _{i}:[0,+\infty )\rightarrow \lbrack 0,+\infty )$ being $C^{1}$
functions satisfying 
\begin{equation*}
(t\phi _{i}(t))^{\prime };\quad \forall t>0\leqno{(\phi_1)}
\end{equation*}%
\begin{equation*}
\lim_{t\rightarrow 0^{+}}t\phi _{i}(t)=0,\quad \lim_{t\rightarrow +\infty
}t\phi _{i}(t)=+\infty \leqno{(\phi_2)}
\end{equation*}%
and that there exist $l_{i},m_{i}\in (1,N),$ $i=0,1,2$ such that%
\begin{equation*}
l_{i}-1\leq \frac{(\phi _{i}(t)t)^{^{\prime }}}{\Phi _{i}(t)}\leq m_{i}-1,%
\text{ \ }\forall t>0,\leqno{(\phi_3)}
\end{equation*}%
where $\phi _{0}:=\phi ,$ $l_{0}:=l$ and $m_{0}:=m.$ Note that the condition 
$(\phi _{3})$ implies that 
\begin{equation*}
l_{i}\leq \frac{\phi _{i}(t)t^{2}}{\Phi _{i}(t)}\leq m_{i},\text{ \ }\forall
t>0,\leqno{(\phi_3)'}
\end{equation*}%
for $i=0,1,2.$ In addition, $\Psi _{i}$ and $\Lambda _{i},$ for $i=0,1,2,$
with $\Psi _{0}:=\Psi $ and $\Lambda _{0}=\Lambda ,$ are $N$-functions
satisfying the $\Delta _{2}$ condition. \newline

According to hypotheses $(\phi _{1})-(\phi _{3})$, a wide class of operators
can be incorporated in problems $(P_{1})$ and $(P_{2})$, for instance:

\begin{itemize}
\item $\phi (t)=p|t|^{p-2},$ $t>0,$ with $p>1.$ The operator $\Delta _{\Phi }
$ is the $p$-Laplacian operator.

\item $\phi (t)=p|t|^{p-2}+q|t|^{q-2},$ $t>0,1<p<q.$ Here $\Delta _{\Phi }$
is the $(p,q)$-Laplacian operator applied in quantum physics (see \cite%
{Benci-Fortunato}).

\item $\phi (t)=2\gamma (1+t^{2})^{\gamma -1},$ $t>0$ and $\gamma >1$. $%
\Delta _{\Phi }$ appears in nonlinear elasticity problems \cite{FN2}.

\item $\phi (t)=\gamma \frac{(\sqrt{1+t^{2}}-1)^{\gamma -1}}{\sqrt{1+t^{2}}},
$ $t>0$ and $\gamma \geq 1$. The operator $\Delta _{\Phi }$ arises in
minimal surfaces theory for $\gamma =1$ (see \cite[page 128]{Daco}) and
nonlinear elasticity for $\gamma >1$ (see \cite{FN0}).

\item $\phi (t)=\frac{pt^{p-2}(1+t)\ln (1+t)+t^{p-1}}{1+t},$ $t>0$. The
operator $\Delta _{\Phi }$ appears in plasticity problems (see \cite{FN2}).%
\newline
\end{itemize}

However, the lack of properties such as homogeneity complicates handling the
nonlinear $\Phi _{i}-$Laplacian operator which, therefore, constitutes a
serious obstacle in the study of the problems $(P_{1})$ and $(P_{2})$.
Thereby, it requires relevant topics of nonlinear functional analysis,
especially theory of Orlicz and Orlicz-Sobolev spaces (see, e.g. \cite{AF,RR}
and their abundant reference). Another mathematical difficulty encountered
comes out from the nonlocal caracter of $(P_{1})$ and $(P_{2})$. It is due
to the presence of terms $|\cdot |_{{L^{\Psi _{i}}}}$ and $|\cdot
|_{L^{\Lambda _{i}}}$ that make the equations in $(P_{1})$ and $(P_{2})$ no
longer a pointwise identities. For more inquiries on nonlocal problems we
refer to \cite{CL-, Chen-Gao} where systems of elliptic equations are
examined. With regard to the scalar case, we quote the papers \cite%
{AC,ACM,CF,kirchhoff,Ma,PZ,ZP}. Such problems are important for applications
in view of the significant number of physical phenomena formulated into
nonlocal mathematical models. For instance, they appear in the study of the
flow of a fluid through a homogeneous isotropic rigid porous medium, as well
as in the study of population dynamics (see, e.g., \cite{DDX, S}).

Relevant contributions regarding nonlocal problems fit the setting of $%
(P_{1})$ and $(P_{2})$. In particular, Alves \& Covei \cite{Alves-Covei}
applied the sub-supersolution method to show the existence results for
problem involving Kirchhoff-type operator 
\begin{equation*}
\left\{ 
\begin{array}{rcl}
-a\left( \int_{\Omega }u\right) \Delta u & = & h_{1}(x,u)f\left(
\int_{\Omega }|u|^{p}\right) +h_{2}(x,u)g\left( \int_{\Omega }|u|^{r}\right)
\;\;\mbox{in}\;\;\Omega , \\ 
u & = & 0\;\;\mbox{on}\;\;\partial \Omega ,%
\end{array}%
\right. 
\end{equation*}%
where $a,f,g,$ and $h_{i}$ ($i=1,2$) are given functions. The case of
nonlocal problems driven by the $p$-Laplacian differential operator is
investigated by \cite{CFL}. Combining sub-supersolutions method with the
classical theorem due to Rabinowitz \cite{rabinowitz}, the authors proved
the the existence of solutions for quasilinear problem of the form 
\begin{equation*}
\left\{ 
\begin{array}{rcl}
-\Delta _{p}u & = & |u|_{L^{q}}^{\alpha (x)}\;\;\mbox{in}\;\;\Omega , \\ 
u & = & 0\;\;\mbox{on}\;\;\partial \Omega ,%
\end{array}%
\right. \eqno{(P)}
\end{equation*}%
where $\alpha $ is a nonnegative function defined in $\overline{\Omega }.$
Then, they extend the results to the nonlocal quasilinear elliptic system 
\begin{equation*}
\left\{ 
\begin{array}{rcl}
-\Delta _{p_{1}}u & = & |v|_{L^{q_{1}}}^{\alpha _{1}(x)}\;\;\mbox{in}%
\;\;\Omega , \\ 
-\Delta _{p_{2}}v & = & |u|_{L^{q_{2}}}^{\alpha _{2}(x)}\;\;\mbox{in}%
\;\;\Omega , \\ 
u & = & 0\;\;\mbox{on}\;\;\partial \Omega .%
\end{array}%
\right. \eqno{(S)}
\end{equation*}%
The semilinear case, that is when $p_{1}=p_{2}=2$ is investigated by Corr%
\^{e}a-Lopes \cite{CL-} and Chen-Gao \cite{Chen-Gao} for systems of the
form 
\begin{equation}
\left\{ 
\begin{array}{rcl}
-\Delta u^{m} & = & f(x,u)|v|_{L^{p}}^{\alpha }\ \text{in}\ \Omega , \\ 
-\Delta v^{n} & = & g(x,v)|u|_{L^{q}}^{\beta }\ \text{in}\ \Omega , \\ 
u & = & v=0\ \text{on}\ \partial \Omega 
\end{array}%
\right.   \label{stat}
\end{equation}%
with $m,n\geq 1$ and $\alpha ,\beta >0$. The existence of solutions is
obtained by means of topological methods, namely, Galerkin method, fixed
point theory as well as sub-supersolutions techniques.

Motivated by the aforementioned papers, our goal is to establish the
existence of (positive) solutions for problems $(P_{1})$ and $(P_{2})$
involving sublinear and concave-convex terms. The approach relies on the
method of sub-supersolution. However, besides the nonlocal nature of the
problems, this method cannot be easily implemented due to the presence of $%
\Phi _{i}$-Laplacian operator in the principle part of the equations. At
this point, to the best of our knowledge, it is for the first time when
nonlocal problems involving $\Phi _{i}$-Laplacian operator are studied. A
significant feature of our result lies in the obtaining of the sub- and
supersolution in the Orlicz-Sobolev spaces setting and, involving nonlocal
terms. This is achieved by the choice of suitable explicit functions with an
adjustment of adequate constants.

The rest of the paper is organized as follows: Section 2 is devoted to the
needed properties in Orlicz and Orlicz-Sobolev spaces. Section 3 (resp.
Section 4) contains existence results for problem $(P_{1})$ (resp. $(P_{2})$%
)\ involving sublinear and concave-convex structures.

\section{Preliminaries}

In this section we recall some results on Orlicz-Sobolev spaces. We say that
a continuous function $\Phi : \mathbb{R} \rightarrow [0,+\infty)$ is a
N-function if:

\begin{itemize}
\item[$(i)$] $\Phi$ is convex,

\item[$(ii)$] $\Phi(t) = 0 \Leftrightarrow t = 0 $,

\item[$(iii)$] $\displaystyle\lim_{t\rightarrow0}\frac{\Phi(t)}{t}=0$ and $%
\displaystyle\lim_{t\rightarrow+\infty}\frac{\Phi(t)}{t}= +\infty$,

\item[$(iv)$] $\Phi$ is even.
\end{itemize}

We say that a N-function $\Phi$ verifies the $\Delta_{2}$-condition, and we
denote by $\Phi \in \Delta_2$, if 
\begin{equation*}
\Phi(2t) \leq K\Phi(t),\quad \forall t\geq t_0,
\end{equation*}
for some constants $K,t_0 > 0$. Regarding the condition $\Delta_2$ it is
important to note that such property is satisfied under the condition $%
(\phi_3)^{^{\prime }}$ in the case of the $N-$function is given by %
\eqref{phi}.

We fix an open set $\Omega \subset \mathbb{R}^{N}$ and a N-function $\Phi$.
We define the Orlicz space associated with $\Phi$ as follows 
\begin{equation*}
L^{\Phi}(\Omega) = \left\{ u \in L^{1}(\Omega) \colon \ \int_{\Omega} \Phi%
\Big(\frac{|u|}{\lambda}\Big)dx < + \infty \ \ \mbox{for some}\ \ \lambda >0
\right\}.
\end{equation*}
The space $L^{\Phi}(\Omega)$ is a Banach space endowed with the Luxemburg
norm given by 
\begin{equation*}
| u |_{L^\Phi} = \inf\left\{ \lambda > 0 : \int_{\Omega}\Phi\Big(\frac{|u|}{%
\lambda}\Big)dx \leq1\right\}.
\end{equation*}


\begin{lemma}
\cite{FN}\label{min-max} Consider $\Phi$ a $N$-function of the form %
\eqref{phi} and satisfying $(\phi_1), (\phi_2)$ and $(\phi_3)$. Set 
\begin{equation*}
\zeta_0(t)=\min\{t^\ell,t^m\}~~\mbox{and}~~ \zeta_1(t)=\max\{t^\ell,t^m\},~~
t\geq 0.
\end{equation*}
Then $\Phi$ satisfies 
\begin{equation*}
\zeta_0(t)\Phi(\rho)\leq\Phi(\rho t)\leq \zeta_1(t)\Phi(\rho),~~ \rho, t> 0,
\end{equation*}
\begin{equation*}
\zeta_0(|u|_{\Phi})\leq\int_\Omega\Phi(u)dx\leq \zeta_1(|u|_{\Phi}),~ u\in
L_{\Phi}(\Omega).
\end{equation*}
\end{lemma}

For a $N-$function $\Phi$, the corresponding Orlicz-Sobolev space is defined
as the Banach space 
\begin{equation*}
W^{1, \Phi}(\Omega) = \Big\{ u \in L^{\Phi}(\Omega) \ :\ \frac{\partial u}{%
\partial x_{i}} \in L^{\Phi}(\Omega), \quad i = 1, ..., N\Big\},
\end{equation*}
endowed with the norm 
\begin{equation*}
\Vert u \Vert_{1,\Phi} = |\nabla u|_{L^\Phi} + |u|_{L^\Phi}.
\end{equation*}

The $\Delta_2$-condition also implies that 
\begin{equation*}
u_n \to u \,\,\, \mbox{in} \,\,\, L_{\Phi}(\Omega) \Longleftrightarrow
\int_{\Omega}\Phi(|u_n-u|) \to 0
\end{equation*}
and 
\begin{equation*}
u_n \to u \,\,\, \mbox{in} \,\,\, W^{1,\Phi}(\Omega) \Longleftrightarrow
\int_{\Omega}\Phi(|u_n-u|) \to 0 \,\,\, \mbox{and} \,\,\,
\int_{\Omega}\Phi(|\nabla u_n- \nabla u|) \to 0.
\end{equation*}

Consider $u,v\in W^{1,\Phi}(\Omega)$ we will say that $-\Delta_{\Phi}u\leq
-\Delta_{\Phi}v$ in $\Omega$ if 
\begin{equation*}
\int_{\Omega}\phi(|\nabla u|)\nabla u
\nabla\varphi\leq\int_{\Omega}\phi(|\nabla v|)\nabla v \nabla\varphi,
\end{equation*}
for all $\varphi\in W_{0}^{1,\Phi}(\Omega)$ with $\varphi\geq0.$

The following results will be often used.

\begin{lemma}
\cite[Lemma 4.1]{Tan-Fang}\label{comparison} Let $u,v \in W^{1,\Phi}(\Omega)$
with $-\Delta_{\Phi}u\leq -\Delta_{\Phi}v$ in $\Omega$ and $u \leq v$ in $%
\partial \Omega$ (i.e $(u-v)^{+} \in W^{1,\Phi}_{0}(\Omega)$), then $u(x)
\leq v(x)$ a.e in $\Omega.$
\end{lemma}


\begin{lemma}
\cite[Lemma 4.5]{Tan-Fang}\label{Tan-Fang} Let $\lambda >0$ and consider $%
z_{\lambda }$ the unique solution of the problem 
\begin{equation}
\begin{aligned} \left\{\begin{array}{rcl} -\Delta_{\Phi}z_{\lambda}
&=&\lambda\;\;\mbox{in} \;\;\Omega,\\ \vspace{.2cm}
z_{\lambda}&=&0\;\;\mbox{on}\;\;\partial\Omega, \end{array} \right.
\end{aligned}  \label{probl-linear-lambda}
\end{equation}%
where $\Phi $ is given by \eqref{phi} and $\Omega \subset \mathbb{R}^{N}$ is
an admissible domain. Define $\rho _{0}=\frac{1}{2|\Omega |^{\frac{1}{N}%
}C_{0}}.$ If $\lambda \geq \rho _{0},$ then $|z_{\lambda }|_{L^{\infty
}}\leq C^{\ast }\lambda ^{\frac{1}{l-1}}$ and $|z_{\lambda }|_{L^{\infty
}}\leq C_{\ast }\lambda ^{\frac{1}{m-1}}$ if $\lambda <\rho _{0}.$ Here $%
C^{\ast }$ and $C_{\ast }$ are positive constants dependending only on $l,m,N
$ and $\Omega $.
\end{lemma}

Regarding to the function $z_{\lambda}$ of the previous result, it follows
from \cite[page 320]{Lieberman} and \cite[Lemma 4.2]{Tan-Fang} that $%
z_{\lambda} \in C^{1}(\overline{\Omega}) $ with $z_\lambda >0$ in $\Omega.$

\section{The scalar case}

We say that $u\in W_{0}^{1,\Phi }(\Omega )\cap L^{\infty }(\Omega )$ is a
(weak) solution of $(P_{1})$ if 
\begin{equation*}
\int_{\Omega }\phi (|\nabla u|)\nabla u\nabla \varphi =\int_{\Omega
}(f(u)|u|_{{\Psi }}^{\alpha }+g(u)|u|_{{\Lambda }}^{\gamma })\varphi ,
\end{equation*}%
for all $\varphi \in W_{0}^{1,\Phi }(\Omega ).$

Given $u,v\in \mathcal{S}(\Omega):= \{u: \Omega \rightarrow \mathbb{R}: u \ 
\text{is measurable}\},$ we write $u \leq v$ if $u(x) \leq v(x)$ a.e in $%
\Omega.$ We denote by $[u,v]$ the set 
\begin{equation*}
[u,v]:=\bigl\{w\in \mathcal{S}(\Omega): u(x)\leq w(x)\leq v(x)\;\text{a.e in}%
\;\Omega\bigl\}.
\end{equation*}

We say that $(\underline{u},\overline{u})$ is a sub-super solution pair for $%
(P)$ if $\underline{u},\overline{u}\in $ $W_{0}^{1,\Phi }(\Omega )\cap
L^{\infty }(\Omega )$ are nonnegative functions that satisfy the inequality $%
0<\underline{u}\leq \overline{u}$ in $\Omega $ and if for all $\varphi \in
W_{0}^{1,\Phi }(\Omega )$ with $\varphi \geq 0$ the following inequalities
hold

\begin{equation*}
\int_{\Omega}\phi(|\nabla \underline{u}|)\nabla\underline{u}%
\nabla\varphi\leq \int_{\Omega}(f(\underline{u})|\underline{u}%
|_{L^{\Psi}}^{\alpha} + g(\underline{u})| \underline{u}|^{\gamma}_{L^{%
\Lambda}})\varphi
\end{equation*}
and 
\begin{equation*}
\int_{\Omega}\phi(|\nabla \overline{u}|)\nabla\overline{u}\nabla\varphi\geq
\int_{\Omega}(f(\overline{u})|\overline{u}|_{L^{\Psi}}^{\alpha} + g(%
\overline{u}) | \overline{u}|^{\gamma}_{L^{\Lambda}})\varphi .
\end{equation*}

The following result will play an important role in our arguments.

\begin{lemma}
\label{sub-supermethod} Suppose that $f,g:[0, +\infty) \rightarrow \mathbb{R}
$ are nondecreasing, continuous and nonnegative functions. Consider also
that $\alpha, \gamma \geq 0$ and that there exists a sub-supersolution pair $%
(\underline{u},\overline{u})$ for problem $(P_1). $ Then there exists a
nontrivial solution $u$ for $(P_1)$ with $u \in [\underline{u},\overline{u}%
]. $
\end{lemma}

\begin{proof}
We have that 
\begin{equation*}
\begin{aligned} \left\{\begin{array}{rcl} -\Delta_{\Phi} \underline{u} &\leq
&f(\underline{u})| \underline{u}|^{\alpha}_{L^\Psi} + g(\underline{u})|
\underline{u}|^{\gamma}_{L^\Lambda}\;\;\mbox{in} \;\;\Omega,\\ \vspace{.2cm}
\underline{u}&=&0\;\;\mbox{on}\;\;\partial\Omega. \end{array} \right.
\end{aligned}
\end{equation*}
and 
\begin{equation*}
\begin{aligned} \left\{\begin{array}{rcl} -\Delta_{\Phi} \overline{u}
&\geq&f(\overline{u})| \overline{u}|^{\alpha}_{L^\Psi} +
g(\overline{u})|\overline{u} |^{\gamma}_{L^{\Lambda}}\;\;\mbox{in}
\;\;\Omega,\\ \vspace{.2cm} \overline{u}&=&0\;\;\mbox{on}\;\;\partial\Omega.
\end{array} \right. \end{aligned}
\end{equation*}

Denote by $u_1$ the unique solution of the problem 
\begin{equation*}
\begin{aligned} \left\{\begin{array}{rcl} -\Delta_{\Phi} u_1 &=&
f(\underline{u})| \underline{u}|^{\alpha}_{L^\Psi} +g(\underline{u})|
\underline{u}|^{\gamma}_{L^\Lambda}\;\;\mbox{in} \;\;\Omega,\\ \vspace{.2cm}
u_1&=&0\;\;\mbox{on}\;\;\partial\Omega. \end{array} \right. \end{aligned}
\end{equation*}
Note that the mentioned solution exist because the term $f(\underline{u})|%
\underline{u}|^{\alpha}_{L^\Psi}$ is bounded. Since $\underline{u} \leq 
\overline{u}$ in $\Omega$, $f$ is nondecreasing and $\alpha,\gamma \geq0$,
we have that $f(\underline{u})| \underline{u}|^{\alpha}_{L^\Psi}\leq f(%
\overline{u})| \overline{u}|^{\alpha}_{L^\Psi}$ and $g(\underline{u})| 
\underline{u}|^{\gamma}_{L^\Psi}\leq g(\overline{u})| \overline{u}%
|^{\gamma}_{L^\Lambda},$ then it follows from Lemma \ref{comparison} that $%
\underline{u} \leq u_1 \leq \overline{u}$ in $\Omega.$

Let $u_{2}$ be the solution of the problem 
\begin{equation*}
\begin{aligned} \left\{\begin{array}{rcl} -\Delta_{\Phi} u_2 &=&f(u_1)|
u_1|^{\alpha}_{L^\Psi} +g(u_1)| u_1|^{\gamma}_{L^\Lambda}\;\;\mbox{in}
\;\;\Omega,\\ \vspace{.2cm} u_2&=&0\;\;\mbox{on}\;\;\partial\Omega.
\end{array} \right. \end{aligned}
\end{equation*}%
Since $\underline{u}\leq u_{1}\leq \overline{u}$ in $\Omega $, we have that 
\begin{equation*}
f(\underline{u})|\underline{u}|_{L^{\Psi }}^{\alpha }+g(\underline{u})|%
\underline{u}|_{L^{\Lambda }}^{\alpha }\leq f(u_{1})|u_{1}|_{L^{\Psi
}}^{\alpha }+g(u_{1})|u_{1}|_{L^{\Lambda }}^{\gamma }\leq f(\overline{u})|%
\overline{u}|_{L^{\Psi }}^{\alpha }+g(\overline{u})|\overline{u}%
|_{L^{\Lambda }}^{\gamma }\ \text{in}\ \Omega .
\end{equation*}%
Thus from Lemma \ref{comparison} we get, 
\begin{equation*}
\underline{u}\leq u_{1}\leq u_{2}\leq \overline{u}\ \text{in}\ \Omega .
\end{equation*}%
Note also that $-\Delta _{\Phi }u_{i}\leq f(\overline{u})|\overline{u}%
|_{L^{\Psi }}^{\alpha }+g(\overline{u})|\overline{u}|_{L^{\Lambda }}^{\gamma
},i=1,2.$ Thus we can construct a sequence $u_{n}$ such that 
\begin{equation*}
\begin{aligned} \left\{\begin{array}{rcl} -\Delta_{\Phi} u_n &=&f(u_{n-1})|
u_{n-1}|^{\alpha}_{L^\Psi} + g(u_{n-1})|
u_{n-1}|^{\gamma}_{L^\Lambda}\;\;\mbox{in} \;\;\Omega,\\ \vspace{.2cm}
u_n&=&0\;\;\mbox{on}\;\;\partial\Omega. \end{array} \right. \end{aligned}%
\eqno{(P_n)}
\end{equation*}%
with $-\Delta _{\Phi }u_{n}\leq f(\overline{u})|\overline{u}|_{L^{\Psi
}}^{\alpha }+g(\overline{u})|\overline{u}|_{L^{\Lambda }}^{\gamma }$ in $%
\Omega $ and $\underline{u}\leq u_{n}\leq \overline{u}$ in $\Omega $ for all 
$n\in \mathbb{N}.$ Using the $C^{1,\alpha }$ estimates up to the boundary
(see \cite{Lieberman}), we have that $u_{n}$ is a bounded sequence in $%
C^{1,\theta }(\overline{\Omega })$ for some $\theta \in (0,1].$ Since the
embedding $C^{1,\theta }(\overline{\Omega })\hookrightarrow C^{1}(\overline{%
\Omega })$ is compact, we can extract a subsequence with $u_{n}\rightarrow u$
in $C^{1}(\overline{\Omega })$ for some $u\in C^{1}(\overline{\Omega })$.
Passing to the limit in $(P_{n})$, we have that $u$ is a nontrivial solution
for problem $(P_{1})$.
\end{proof}

\subsection{A sublinear scalar problem}

In this section we use Lemma \ref{sub-supermethod} and a suitable
sub-supersolution pair to prove the existence of solution for a nonlocal
problem of the type 
\begin{equation*}
\left\{ 
\begin{array}{rcl}
-\Delta _{\Phi }u & = & u^{\beta }|u|_{L^{\Psi }}^{\alpha }\ \mbox{in}\
\Omega , \\ 
u & = & 0\ \mbox{on}\ \partial \Omega ,%
\end{array}%
\right. \eqno{(P_S)}
\end{equation*}%
where $\alpha ,\beta \geq 0$ are constants saisfying certain conditions. The
above problem is considered in \cite{CFL} for the $p-$Laplacian case and
with $\beta =0$. We complete the study done in \cite[Theorem 4.1]{CFL} by
considering constants exponents and a more general operator.

\begin{theorem}
\label{teo-sublinear} Suppose that $\alpha, \beta \geq 0$ with $0 < \alpha +
\beta < l-1,$ where $l$ is given in $(\phi_3).$ Then $(P_S)$ has a positive
solution.
\end{theorem}

\begin{proof}
We will start by constructing $\overline{u}$. Let $\lambda>0$ and consider $%
z_{\lambda}\in W_{0}^{1,\Psi}(\Omega)\cap L^{\infty}(\Omega)$ the unique
solution of \eqref{probl-linear-lambda} where $\lambda$ will be chosen later.

For $\lambda >0$ large by Lemma \ref{Tan-Fang} there is a constant $K>1$
that does not depend on $\lambda$ such that 
\begin{equation}  \label{desig1-p-supsol}
0<z_{\lambda}(x)\leq K\lambda^{\frac{1}{l-1}}\;\text{in}\;\Omega.
\end{equation}

Since $0<\alpha+\beta<l-1$ we can choose $\lambda>1$ such that %
\eqref{desig1-p-supsol} occurs and 
\begin{equation}  \label{desig2-p-supsol}
K^{\beta}\lambda^{\frac{\alpha+\beta}{l-1}} |K|^{\alpha}_{L^\Psi}\leq\lambda.
\end{equation}

By \eqref{desig1-p-supsol} and \eqref{desig2-p-supsol} we get 
\begin{equation*}
z_{\lambda}^{\beta}|z_{\lambda}|_{L^{\Psi}}^{\alpha}\leq \lambda.
\end{equation*}
Therefore 
\begin{equation*}
\begin{aligned} \left\{\begin{array}{rcl}
-\Delta_{\Phi}z_{\lambda}&\geq&z_{\lambda}^{\beta}|z_{\lambda}|_{{\Psi}}^{%
\alpha}\;\;\mbox{in}\;\;\Omega,\\ \vspace{.2cm}
z_{\lambda}&=&0\;\;\mbox{on}\;\;\partial\Omega. \end{array} \right.
\end{aligned}
\end{equation*}
Now we will construct $\underline{u}.$ Since $\partial \Omega$ is $C^2$
there is a constant $\delta >0$ such that $d \in C^{2}(\overline{ \Omega_{3
\delta}})$ and $|\nabla d(x)| \equiv 1$ where $d(x):= dist(x,\partial
\Omega) $ and $\overline{ \Omega_{3 \delta}}:=\{x \in \overline{\Omega};
d(x) \leq 3 \delta\}$(see \cite[Lemma 14.16]{Gilbarg-Trudinger} and its
proof). Let $\sigma \in (0, \delta)$. A direct computation implies that the
function $\phi=\phi(k,\sigma)$ defined by 
\begin{equation*}
\eta(x)=\left\{%
\begin{array}{lcl}
e^{kd(x)}-1 & \text{ if } & d(x)<\sigma, \\ 
e^{k\sigma}-1+\int_{\sigma}^{d(x)}ke^{k\sigma}\Big(\frac{2\delta-t}{%
2\delta-\sigma}\Big)^{\frac{m}{l-1}}dt & \text{ if } & \sigma\leq
d(x)<2\delta, \\ 
e^{k\sigma}-1+\int_{\sigma}^{2\delta}ke^{k\sigma}\Big(\frac{2\delta-t}{%
2\delta-\sigma}\Big)^{\frac{m}{l-1}}dt & \text{ if } & 2\delta \leq d(x)%
\end{array}
\right.
\end{equation*}
belongs to $C^{1}_{0}(\overline{\Omega})$ where $k>0$ is an arbitrary
number. Direct computations implies that 
\begin{equation*}
-\Delta_{\Phi}(\mu\eta)=%
\begin{cases}
- \mu k^2 e^{kd(x)} \frac{d}{dt} \left( \phi(t)t\right)\bigg\vert_{t= \mu k
e^{kd(x)}} - \phi(\mu k e^{kd(x)})\mu k e^{kd(x)}\Delta d \;\; \mbox{ if}%
\quad d(x)<\sigma, \\ 
\mu k e^{k \sigma}\left( \frac{m}{l-1}\right)\left(\frac{2\delta -d(x)}{%
2\delta- \sigma} \right)^{\frac{m}{l-1}-1}\left( \frac{1}{2\delta - \sigma}%
\right) \frac{d}{dt}\left( \phi(t)t\right)\bigg\vert_{t = \mu k e^{k \sigma}
\left( \frac{2\delta -d(x)}{2\delta - \sigma}\right)} \\ 
-\phi\left( \mu k e^{k\sigma} \left( \frac{2\delta -d(x)}{2\delta - \sigma}%
\right)^{\frac{m}{l-1}}\right) \mu k e^{k\sigma} \left( \frac{2\delta-d(x)}{%
2\delta - \sigma}\right)^{\frac{m}{l-1}}\Delta d \;\; \mbox{ if}\quad \sigma
< d(x)<2\delta, \\ 
0\;\; \mbox{ if}\quad 2\delta<d(x)%
\end{cases}%
\end{equation*}
for all $\mu >0$.

If $k$ is large and $d(x)<\sigma,$ we have that $-\Delta_\Phi(\mu \phi)\leq
0.$ In fact, note that by $(\phi_3)$ we have for $k$ large that 
\begin{equation}  \label{negative}
\begin{aligned} -\Delta_{\Phi} (\mu \eta) =& - \mu k^2 e^{kd(x)}
\frac{d}{dt}\left( \phi(t)t\right)\bigg\vert_{t= \mu k e^{kd(x)}} - \phi(\mu
k e^{kd(x)})\mu k e^{kd(x)}\Delta d \\ \leq &- k^2 \mu
e^{kd(x)}(l-1)\phi(\mu k e^{kd(x)}) - \phi(\mu k e^{kd(x)})\mu k
e^{kd(x)}\Delta d \\ =& \mu ke^{kd(x)}\phi (\mu k e^{kd(x)})(-k(l-1) -
\Delta d)\\ \leq & 0, \end{aligned}
\end{equation}
because $\Delta d$ is bounded near the boundary and $l>1.$

Now we will estimate $-\Delta_{\Phi} (\mu \eta)$ in the case $\sigma < d(x)
< 2\delta.$ Note that from $(\phi_3)$ and Lemma \ref{min-max} we get {\small 
\begin{equation}  \label{est1}
\begin{aligned} \mu & k e^{k \sigma}\left(
\frac{m}{l-1}\right)\left(\frac{2\delta -d(x)}{2\delta- \sigma}
\right)^{\frac{m}{l-1}-1}\left( \frac{1}{2\delta - \sigma}\right)
\frac{d}{dt}\left( \phi(t)t\right)\bigg\vert_{t = \mu k e^{k \sigma} \left(
\frac{2\delta -d(x)}{2\delta - \sigma}\right)} \\ \leq &\mu k e^{k
\sigma}\left( \frac{m}{l-1}\right)\left(\frac{2\delta -d(x)}{2\delta-
\sigma} \right)^{\frac{m}{l-1}-1}\left( \frac{m-1}{2\delta - \sigma}\right)
\phi \left( \mu k e^{k \sigma} \left( \frac{2\delta -d(x)}{2\delta -
\sigma}\right)^{\frac{m}{l-1}}\right)\\ \leq &
\left(\frac{m-1}{2\delta-\sigma}\right) \left(\frac{m}{l-1} \right)
\frac{\Phi \left( \mu k e^{k\sigma} \left( \frac{2\delta
-d(x)}{2\delta-\sigma}\right)^{\frac{m}{l-1}}\right)}{\mu k e^{k\sigma}
\left( \frac{2\delta -d(x)}{2\delta -\sigma}\right)^{\frac{m}{l-1}}}
\frac{1}{\left(\frac{2\delta -d(x)}{2\delta - \sigma}\right)}\\ \leq & \max
\left\{ (\mu k e^{k\sigma})^{m-1} \left( \frac{2\delta - d(x)}{2\delta
-\sigma}\right)^{m \left( \frac{m}{l-1}\right) - \left(\frac{m}{l-1} +1
\right)}, (\mu k e^{k\sigma})^{l-1} \left( \frac{2\delta - d(x)}{2\delta
-\sigma}\right)^{l \left( \frac{m}{l-1}\right) - \left(\frac{m}{l-1} +1
\right)}\right\}\\ \times & \left( \frac{m-1}{2\delta -\sigma}\right) \left(
\frac{m}{l-1}\right). \end{aligned}
\end{equation}
} Since $m,l >1$, we get $l \left(\frac{m}{l-1} \right) - m \left(\frac{m}{%
l-1} +1\right), m \left(\frac{m}{l-1} \right) - m \left(\frac{m}{l-1}
+1\right)> 0. $ Note that $0 \leq \left( \frac{2\delta - d(x)}{2\delta -
\sigma}\right) \leq 1.$ Thus by \eqref{est1} we get 
\begin{equation}  \label{est2}
\begin{aligned} \mu & k e^{k \sigma}\left(
\frac{m}{l-1}\right)\left(\frac{2\delta -d(x)}{2\delta- \sigma}
\right)^{\frac{m}{l-1}-1}\left( \frac{1}{2\delta - \sigma}\right)
\frac{d}{dt}\left( \phi(t)t\right)\bigg\vert_{t = \mu k e^{k \sigma} \left(
\frac{2\delta -d(x)}{2\delta - \sigma}\right)}\\ \leq & \left(
\frac{m-1}{2\delta-\sigma}\right) \left( \frac{m}{l-1}\right) \max \{(\mu k
e^{k\sigma})^{m-1}, (\mu k e^{k\sigma})^{l-1}\} \\ =& C_1 \left(
\frac{1}{2\delta -\sigma}\right) \max \{(\mu k e^{k\sigma})^{m-1}, (\mu k
e^{k\sigma})^{l-1}\}, \end{aligned}
\end{equation}
where $C_1$ is a constant that does not depend on $\mu $ and $k.$ On other
hand, we have by Lemma \ref{min-max} that {\small 
\begin{equation}  \label{est3}
\begin{aligned} &\left|\phi \left(\mu k e^{k\sigma} \left( \frac{2\delta
-d(x)}{2\delta - \sigma}\right)^{\frac{m}{l-1}} \right) \mu k e^{k\sigma}
\left( \frac{2\delta - d(x)}{2\delta - \sigma}\right)^{\frac{m}{l-1}} \Delta
d\right|\\ \leq & \phi \left(\mu k e^{k\sigma} \left( \frac{2\delta
-d(x)}{2\delta - \sigma}\right)^{\frac{m}{l-1}} \right) \mu k e^{k\sigma}
\left( \frac{2\delta - d(x)}{2\delta - \sigma}\right)^{\frac{m}{l-1}}
\displaystyle\sup_{\overline{\Omega_{3\delta}} } |\Delta d| \\ \leq & C
\frac{\Phi \left( \mu k e^{k\sigma} \left(\frac{2\delta - d(x)}{2\delta -
\sigma} \right)^{\frac{m}{l-1}}\right)}{\mu k e^{k\sigma} \left(
\frac{2\delta - d(x)}{2\delta - \sigma}\right)^{\frac{m}{l-1}}}\\ \leq & C
\max \left\{ (\mu k e^{k\sigma})^{m-1} \left( \frac{2\delta - d(x)}{2\delta
-\sigma}\right)^{m \left( \frac{m}{l-1}\right) - \left(\frac{m}{l-1} +1
\right)}, (\mu k e^{k\sigma})^{l-1} \left( \frac{2\delta - d(x)}{2\delta
-\sigma}\right)^{l \left( \frac{m}{l-1}\right) - \left(\frac{m}{l-1} +1
\right)}\right\}\\ \leq & C_2 \max\{(\mu k e^{k\sigma})^{m-1}, (\mu k
e^{k\sigma})^{l-1}\}, \end{aligned}
\end{equation}
} where $C_2$ is a constant that does not depend on $\sigma, k$ and $\mu.$
Thus from \eqref{est2} and \eqref{est3} we have that 
\begin{equation}  \label{midle}
- \Delta_{\Phi} u \leq \max\left\{ \frac{C_1 }{2\delta-\sigma}, C_2 \right\}
\max\{(\mu k e^{k\sigma})^{m-1}, (\mu k e^{k\sigma})^{l-1}\},
\end{equation}
if $\sigma < d(x) < 2\delta.$

Consider the function $\eta$ and the numbers $\mu, \sigma$ and $k>0$
described before. Let $\sigma = \frac{\ln 2}{k}$ and $\mu = e^{-k}.$ Then $%
e^{k\sigma} =2 $.

If $k>0$ is large, we have from \eqref{negative} that 
\begin{equation}  \label{f1}
-\Delta_{\Phi}(\mu \eta) \leq 0 \leq (\mu \eta)^{\beta}| \mu
\eta|^{\alpha}_{L^\Psi}
\end{equation}
in the case $d(x) < \sigma.$

For any $k>0$ we have $\eta (x)\geq e^{k\sigma }-1=2-1=1$ in $\Omega .$ Thus
there is a constant $C_{3}>0$ that does not depend on $k>0$ such that 
\begin{equation*}
(\mu \eta )^{\beta }|\mu \eta |_{L^{\Psi }}^{\alpha }\geq \mu ^{\alpha
+\beta }C_{3}
\end{equation*}%
Since $0<\alpha +\beta <l-1$, the L'Hospital's rule implies that 
\begin{equation*}
\lim_{k\rightarrow +\infty }\frac{k^{l-1}}{e^{k(l-1-(\alpha +\beta ))}}=0.
\end{equation*}%
Thus, it is possible to consider a large $k_{0}>0$ such that 
\begin{equation*}
C_{3}\geq \max \left\{ C_{1}\frac{1}{2\delta -\frac{\ln 2}{k}},C_{2}\right\}
\max \{2^{m-1},2^{l-1}\}\frac{k^{l-1}}{e^{k(l-1-(\alpha +\beta ))}},
\end{equation*}%
for all $k\geq k_{0}.$ From $\eqref{midle},$ we have that 
\begin{equation}
-\Delta _{\Phi }(\mu \eta )\leq (\mu \eta )^{\beta }|\mu \eta |_{L^{\Psi
}}^{\alpha }  \label{f2}
\end{equation}%
in the region $\sigma <d(x)<2\delta $ for $k>0$ is large enough.

If $d(x)> 2\delta$ we have 
\begin{equation}  \label{f3}
- \Delta_{\Phi}(\mu \eta) = 0 \leq (\mu \eta)^{\beta} |\mu \eta
|^{\alpha}_{L^\Psi}.
\end{equation}
Thus from \eqref{f1}, \eqref{f2} and \eqref{f3} we have that $\mu \eta $ is
a subsolution for $(Ps).$ Note that from $\eqref{midle}, \eqref{f1}$ and $%
\eqref{f3}$ we have for $k,\lambda>0 $ large enough that $-\Delta_{\Phi}
(\mu \eta) \leq -\Delta_{\Phi} z_{\lambda}$. Thus from Lemma \ref{Tan-Fang}
we have $\mu \eta \leq z_{\lambda}$ in $\Omega.$ From Lemma \ref%
{sub-supermethod} we have the result.
\end{proof}

\begin{remark}
An interesting question for problem $(P_S)$ is the existence of solution in
the case $l-1 < \alpha + \beta.$
\end{remark}

\subsection{A concave-convex scalar problem:}

In this section we will consider a concave-convex problem of the type

\begin{equation*}
\left\{ 
\begin{array}{ll}
-\Delta _{\Phi }u=\lambda u^{\beta }|u|_{L^{\Psi }}^{\alpha }+\theta u^{\xi
}|u|_{L^{\Lambda }}^{\gamma } & \mbox{in}\;\;\Omega , \\ 
u=0 & \mbox{on}\;\;\partial \Omega ,%
\end{array}%
\right. \eqno{(P)_{\lambda,\theta}}
\end{equation*}%
where $\alpha ,\beta ,\xi ,\gamma \geq 0$ are constants satisfying certain
conditions and $\lambda ,\theta >0$ are positive numbers. The local version
of $(P)_{\lambda ,\theta }$ for the Laplacian operator was considered in the
famous paper by Ambrosetti-Brezis-Cerami \cite{ABC} in which a
sub-supersolution argument is used. Our result is the following one.


\begin{theorem}
Suppose that $\alpha,\beta,\xi,\gamma \geq 0$ and consider also that $0 <
\alpha +\beta < l-1 $. The following assertions hold.

\vspace{0.2cm}

\noindent\textbf{$(i)$} If $m-1<\xi+\gamma,$ then given $\theta>0$ there
exists $\lambda_{0}>0$ such that for each $\lambda\in(0,\lambda_{0})$ the
problem $(P)_{\lambda,\theta}$ has a positive solution $u_{\lambda,\theta}.$

\vspace{0.2cm}

\noindent\textbf{$(ii)$} If $l-1<\xi+\gamma$, then given $\lambda>0$ there
exists $\theta_{0}>0$ such that for each $\theta\in(0,\theta_{0})$ the
problem $(P)_{\lambda,\theta}$ has a positive solution $u_{\lambda,\theta}.$
\end{theorem}

\begin{proof}
Suppose that $(i)$ occurs and fix $\theta >0.$ Let $z_{\lambda }\in
W_{0}^{1,\Phi }({\Omega })\cap L^{\infty }(\Omega )$ be the unique solution
of \eqref{probl-linear-lambda} where $\lambda \in (0,1)$ will be chosen
before.

Lemma \ref{Tan-Fang} implies that for $\lambda>0$ small enough there exists
a constant $K>1$ that does not depend on $\lambda$ such that 
\begin{equation}  \label{desig1-p-supsol-concavo}
0<z_{\lambda}(x)\leq K\lambda^{\frac{1}{m-1}}\;\text{in}\;\Omega.
\end{equation}

Let $\overline{K}:=\max \big\{K^{\beta }|K|_{L^{\Psi }}^{\alpha },K^{\xi
}|K|_{L^{\Lambda }}^{\gamma }\}.$ For each $\theta >0$ we can choose $%
0<\lambda _{0}<1$ small enough, depending on $\theta ,$ such that the
inequalities 
\begin{equation*}
\lambda \geq \left( \lambda ^{\frac{\alpha +\beta +m-1}{m-1}}\overline{K}%
+\theta \overline{K}\lambda ^{\frac{\xi +\gamma }{m-1}}\right) ,\ \text{for
all}\ \lambda \in (0,\lambda _{0})
\end{equation*}%
and \eqref{desig1-p-supsol-concavo} hold because $\alpha +\beta >0$ and $%
m-1<\xi +\gamma .$ Thus, there is a small $\lambda _{0}>0$ such that 
\begin{align*}
(\lambda z_{\lambda }^{\beta }|z_{\lambda }|_{L^{\Psi }}^{\alpha }+\theta
z_{\lambda }^{\xi }|z_{\lambda }|_{L^{\Lambda }}^{\gamma })& \leq \lambda
(K\lambda ^{\frac{1}{m-1}})^{\beta }|K\lambda ^{\frac{1}{m-1}}|_{L^{\Psi
}}^{\alpha } \\
& +\theta (K\lambda ^{\frac{1}{m-1}})^{\xi }|K\lambda ^{\frac{1}{m-1}%
}|_{L^{\Lambda }}^{\gamma } \\
& \leq \lambda .
\end{align*}%
for all $\lambda \in (0,\lambda _{0}).$ Thus for $\lambda \in (0,\lambda
_{0})$ we get%
\begin{equation*}
\lambda z_{\lambda }^{\beta }|z_{\lambda }|_{L^{\Psi }}^{\alpha }+\theta
z_{\lambda }^{\xi }|z_{\lambda }|_{L^{\Lambda }}^{\gamma }\leq \lambda .
\end{equation*}

Now consider $\eta, \delta,\sigma,\mu$ and as in the proof of Theorem \ref%
{teo-sublinear}. Fix $\lambda\in(0,\lambda_{0})$.

Since $\alpha+\beta<l-1$ the arguments of the proof of Theorem \ref%
{teo-sublinear} implies that if $\mu=\mu(\lambda)>0$ is small enough then 
\begin{equation*}
-\Delta_{\Phi}(\mu \eta)\leq\lambda \ \text{in} \ \Omega
\end{equation*}
and 
\begin{align*}
-\Delta_{\Phi}(\mu\eta)&\leq
\lambda(\mu\eta)^{\beta}|\mu\eta|_{L^{\Psi}}^{\alpha} \\
&\leq \lambda(\mu\eta)^{\beta}|\mu\eta|_{L^{\Psi}}^{\alpha} +
\lambda(\mu\eta)^{\xi}|\mu\eta|_{L^{\Lambda}}^{\gamma}.
\end{align*}
The weak comparison principle implies that $\mu \eta \leq z_{\lambda}$ for $%
\mu = \mu(\lambda) >0$ small enough. Therefore $(\mu \eta, z_{\lambda})$ is
a sub-super solution pair for $(P)_{\lambda, \theta}.$

Now we will prove the theorem in the second case. Consider again $\eta
,\delta ,\sigma $ and $\mu $ as in the proof of Theorem \ref{teo-sublinear}.
Let $\lambda \in (0,\infty )$. Since $\alpha +\beta <l-1$ we can repeat the
arguments of Theorem \ref{teo-sublinear} to obtain $\mu =\mu (\lambda )>0$
small depending only on $\lambda $ such that%
\begin{equation*}
-\Delta _{\Phi }(\mu \eta )\leq 1\;\;\text{and}\;\;-\Delta _{\Phi }(\mu \eta
)\leq \lambda (\mu \eta )^{\beta }|\mu \eta |_{L^{\Psi }}^{\alpha }\;\;\text{%
in}\;\Omega .
\end{equation*}

Let $z_{M}\in W_{0}^{1,\Phi}(\Omega)\cap L^{\infty}(\Omega)$ the unique
solution of $\eqref{probl-linear-lambda}$ where $M>0$ will be chosen later.

For $M\geq 1$ large enough there is a constant $K>1$ that does not depend on 
$M$ such that 
\begin{equation}  \label{desig2-p-supsol-concavo}
0<z_{M}(x)\leq KM^{\frac{1}{l-1}}\;\text{in}\;\Omega.
\end{equation}

We want to obtain $M>1$ such that 
\begin{equation}  \label{eqnova}
M\geq\left(\lambda z_{M}^{\beta}|z_{M}|_{L^{\Psi}}^{\alpha}+\theta
z_{M}^{\xi}|z_{M}|_{L^{\Lambda}}^{\gamma}\right)\;\mbox{in}\;\Omega
\end{equation}
occurs.

Denoting by $I$ the right-hand side of \eqref{eqnova}, we have from %
\eqref{desig2-p-supsol-concavo} that $I\leq M$ if%
\begin{equation}
1\geq \lambda \overline{K}M^{\frac{\alpha +\beta }{l-1}-1}+\theta \overline{K%
}M^{\frac{\xi +\gamma }{l-1}-1},  \label{equiv-rel2}
\end{equation}%
where $\overline{K}:=\max \{K^{\beta }|K|_{L^{\Psi }}^{\alpha },K^{\xi
}|K|_{L^{\Psi }}^{\gamma }\}.$ Since $0<\alpha +\beta <l-1<\xi +\gamma ,$
the function 
\begin{equation*}
\Psi (t)=\lambda \overline{K}t^{\rho -1}+\theta \overline{K}t^{\tau -1},t>0,
\end{equation*}%
where $\rho :=\frac{\alpha +\beta }{l-1}$ and $\tau :=\frac{\xi +\gamma }{l-1%
},$ belongs to $C^{1}\big((0,\infty ),\mathbb{R}\big)$ and attains a global
minimum at%
\begin{equation}
M_{\lambda ,\theta }:=M(\lambda ,\theta )=L\Biggl(\dfrac{\lambda }{\theta }%
\Biggl)^{\frac{1}{\tau -\rho }}  \label{minimum}
\end{equation}%
where $L:=(\frac{1-\rho }{\tau -1})^{\frac{1}{\tau -\rho }}.$ The inequality %
\eqref{equiv-rel2} is equivalent to find $M_{\lambda ,\theta }>0$ such that $%
\Psi (M_{\lambda ,\theta })\leq 1.$ By \eqref{minimum} we have $\Psi
(M_{\lambda ,\theta })\leq 1$ if and only if 
\begin{equation*}
\lambda \overline{K}(1-\rho )^{\frac{\rho -1}{\tau -\rho }}\left( \frac{%
\lambda }{\theta }\right) ^{\frac{\rho -1}{\tau -\rho }}+\theta \overline{K}%
(1-\rho )^{\frac{\tau -1}{\tau -\rho }}\left( \frac{\lambda }{\theta }%
\right) ^{\frac{\tau -1}{\tau -\rho }}\leq 1
\end{equation*}%
Notice that the above inequality holds if $\theta >0$ is small enough
because $\alpha +\beta <l-1<\xi +\gamma $. Thus for $\lambda >0$ fixed there
exists $\theta _{0}=\theta _{0}(\lambda )$ such that for each $\theta \in
(0,\theta _{0})$ there is a number $M=M_{\lambda ,\theta }>0$ such that %
\eqref{equiv-rel2} occurs. Consequently we have \eqref{eqnova}. Therefore 
\begin{equation*}
-\Delta _{\Phi }z_{M}\geq \lambda z_{M}^{\beta }|z_{M}|_{L^{\Psi }}^{\alpha
}+\theta z_{M}^{\xi }|z_{M}|_{L^{\Lambda }}^{\gamma }\ \text{in}\ \Omega .
\end{equation*}%
Considering if necessary a smaller $\theta _{0}>0,$ we get $M\geq 1$ .
Therefore $-\Delta _{\Phi }(\mu \eta )\leq -\Delta _{\Phi }z_{M}$ in $\Omega
.$ The weak comparison principle implies that $\mu \eta \leq z_{M}$. Then $%
(\mu \eta ,z_{M})$ is a sub-supersolution pair for $(P)_{\lambda ,\theta }.$
The proof is finished.
\end{proof}

\section{The system case}

We say that $(u_{1},u_{2})\in (W_{0}^{1,\Phi _{1}}(\Omega )\cap L^{\infty
}(\Omega ))\times (W_{0}^{1,\Phi _{2}}(\Omega )\cap L^{\infty }(\Omega ))$
is a (weak) solution of $(P_{2})$ if 
\begin{equation*}
\int_{\Omega }\phi (|\nabla u_{i}|)\nabla u_{i}\nabla \varphi =\int_{\Omega
}(f_{i}(u_{j})|u_{j}|_{L^{\Psi _{i}}}^{\alpha
_{i}}+g_{i}(u_{j})|u_{j}|_{L^{\Lambda _{i}}}^{\alpha _{i}})\varphi _{i},
\end{equation*}%
for all $\varphi _{i}\in W_{0}^{1,\Phi _{i}}(\Omega )$ with $i,j=1,2$ and $%
i\neq j.$

We say that the pairs $(\underline{u}_{i},\overline{u}_{i}),i=1,2$ are
sub-supersolution pairs for $(P_{2})$ if $\underline{u}_{i},\overline{u}%
_{i}\in W_{0}^{1,\Phi _{i}}(\Omega )\cap L^{\infty }(\Omega )$ are
nonnegative functions with $0<\underline{u}_{i}\leq \overline{u}_{i}$ in $%
\Omega $ and if for all $\varphi _{i}\in W_{0}^{1,\Phi _{i}}(\Omega )$ with $%
\varphi _{i}\geq 0$ the following inequalities are verified 
\begin{equation}
\left\{ 
\begin{array}{r}
\displaystyle\int_{\Omega }\phi _{i}(|\nabla \underline{u}_{i}|)\nabla 
\underline{u}_{i}\nabla \varphi _{i}\leq \displaystyle\int_{\Omega }\left(
f_{i}(\underline{u}_{j})|\underline{u}_{j}|_{L^{\Psi _{i}}}^{\alpha
_{i}}+g_{i}(\underline{u}_{j})|\underline{u}_{j}|_{L^{\Lambda _{i}}}^{\gamma
_{i}}\right) \varphi _{i}, \\ 
\displaystyle\int_{\Omega }\phi _{i}(|\nabla \overline{u}_{i}|)\nabla 
\overline{u}_{i}\nabla \varphi _{i}\geq \displaystyle\int_{\Omega }\left(
f_{i}(\overline{u}_{j})|\overline{u}_{j}|_{L^{\Psi _{i}}}^{\alpha
_{i}}+g_{i}(\overline{u}_{j})|\overline{u}_{j}|_{L^{\Lambda _{i}}}^{\gamma
_{i}}\right) \varphi _{i},%
\end{array}%
\right.  \label{eq2.1}
\end{equation}%
for all $\varphi _{i}\in W_{0}^{1,\Phi _{i}}(\Omega )$ with $i,j=1,2$ and $%
i\neq j.$

The following lemma is needed to obtain a solution for system $(P_2).$

\begin{lemma}
\label{sub-supermethod-sys} Suppose that $f_i,g_i:[0, +\infty) \rightarrow 
\mathbb{R}, i=1,2$ are nondecreasing, continuous and nonnegative functions.
Consider also that $\alpha_i, \gamma_i \geq 0, i =1,2$ and that there exist
sub-supersolution pairs $(\underline{u}_i,\overline{u}_i), i=1,2$ for $%
(P_2). $ Then there exists a solution $(u,\widetilde{u})$ for $(P_2)$ with $%
u \in [\underline{u}_1,\overline{u}_1]$ and $\widetilde{u} \in [\underline{u}%
_2,\overline{u}_2].$
\end{lemma}

\begin{proof}
Consider $u_1$ the solution of the problem

\begin{equation*}
\begin{aligned} \left\{\begin{array}{rcl} -\Delta_{\Phi_1} u_1 &=&
f_1(\underline{u}_2)| \underline{u}_2|^{\alpha_1}_{L^{\Psi_1}}
+g_1(\underline{u}_2)|
\underline{u}_2|^{\gamma_1}_{L^{\Lambda_1}}\;\;\mbox{in} \;\;\Omega,\\
\vspace{.2cm} u_1&=&0\;\;\mbox{on}\;\;\partial\Omega. \end{array} \right.
\end{aligned}
\end{equation*}
Using the monotonicity of $f_1, g_1$ and the fact that $\underline{u}_2 \leq 
\overline{u}_2$ a.e in $\Omega$ we get 
\begin{equation*}
- \Delta_{\Phi_1} \overline{u}_1 \geq f_1(\overline{u}_2)|\overline{u}%
_2|^{\alpha_1}_{L^{\Psi_1}} + g_1(\overline{u}_2) |\overline{u}%
_2|^{\gamma_1}_{L^{\Lambda_1}} \geq -\Delta_{\Phi_1}u_1 \ \text{in} \ \Omega,
\end{equation*}
therefore $u_1 \leq \overline{u}_1.$ Note also that 
\begin{equation*}
-\Delta_{\Phi_1} u_1 = f_1(\underline{u}_2)|\underline{u}_2|^{\alpha_1}_{L^{%
\Psi_1}} + g_1(\underline{u}_2) |\underline{u}_2|^{\gamma_1}_{L^{\Lambda_1}}
\geq -\Delta_{\Phi_1} \underline{u}_1 \ \text{in} \ \Omega.
\end{equation*}
Therefore $\underline{u}_1 \leq u_1 \leq \overline{u}_1$ a.e in $\Omega.$
Denote by $\widetilde{u}_1$ the weak solution of the problem 
\begin{equation*}
\begin{aligned} \left\{\begin{array}{rcl} -\Delta_{\Phi_1} \widetilde{u}_1
&=& f_2(\underline{u}_1)| \underline{u}_1|^{\alpha_2}_{L^{\Psi_2}}
+g_2(\underline{u}_1)|
\underline{u_1}|^{\gamma_2}_{L^{\Lambda_2}}\;\;\mbox{in} \;\;\Omega,\\
\vspace{.2cm} \widetilde{u}_1&=&0\;\;\mbox{on}\;\;\partial\Omega.
\end{array} \right. \end{aligned}
\end{equation*}
From the definition of $\underline{u}_2$ and $\overline{u}_2$ we have that $%
- \Delta_{\Phi_2} \underline{u}_2 \leq-\Delta_{\Phi_2} \widetilde{u}_1 \leq
- \Delta_{\Phi_2} \overline{u}_2$ in $\Omega.$ Therefore $\underline{u}_2
\leq \widetilde{u}_1 \leq \overline{u}_2$ in $\Omega.$

Consider $u_2$ the solution of the problem 
\begin{equation*}
\begin{aligned} \left\{\begin{array}{rcl} -\Delta_{\Phi_1} u_2 &=&
f_1(\widetilde{u}_1)| \widetilde{u}_1|^{\alpha_1}_{L^{\Psi_1}}
+g_1(\widetilde{u}_1)|
\widetilde{u}_1|^{\gamma_1}_{L^{\Lambda_1}}\;\;\mbox{in} \;\;\Omega,\\
\vspace{.2cm} u_2&=&0\;\;\mbox{on}\;\;\partial\Omega. \end{array} \right.
\end{aligned}
\end{equation*}
Using the fact that $\underline{u}_2 \leq \widetilde{u}_1 \leq \overline{u}%
_2 $ in $\Omega$ and the monotonicity of the functions $f_1$ and $g_1,$ we
have $- \Delta_{\Phi_1} u_1 \leq - \Delta_{\Phi_1} u_2 \leq -
\Delta_{\Phi_1} \overline{u}_1$ in $\Omega.$ Therefore $\underline{u}_1 \leq
u_1 \leq u_2 \leq \overline{u}_1$ in $\Omega.$

Consider $\widetilde{u}_2$ the solution of the problem 
\begin{equation*}
\begin{aligned} \left\{\begin{array}{rcl} -\Delta_{\Phi_2} \widetilde{u}_2
&=& f_2(u_1)| u_1|^{\alpha_2}_{L^{\Psi_2}} +g_2(\widetilde{u}_1)|
u_1|^{\gamma_2}_{L^{\Lambda_2}}\;\;\mbox{in} \;\;\Omega,\\ \vspace{.2cm}
\widetilde{u}_2&=&0\;\;\mbox{on}\;\;\partial\Omega. \end{array} \right.
\end{aligned}
\end{equation*}
A direct computation imply that $\underline{u}_2 \leq\widetilde{u}_1 \leq 
\widetilde{u}_2 \leq \overline{u}_2$ in $\Omega.$ Proceeding with the
previous reasonings we construct sequences $u_n$ and $\widetilde{u}_n$
satisfying 
\begin{equation*}
\begin{aligned} \left\{\begin{array}{rcl} -\Delta_{\Phi_1} u_n &=&
f_1(\widetilde{u}_{n-1})| \widetilde{u}_{n-1}|^{\alpha_1}_{L^{\Psi_1}}
+g_1(\widetilde{u}_{n-1})|
\widetilde{u}_{n-1}|^{\gamma_1}_{L^{\Lambda_1}}\;\;\mbox{in} \;\;\Omega,\\
\vspace{.2cm} \widetilde{u}_n&=&0\;\;\mbox{on}\;\;\partial\Omega.
\end{array} \right. \end{aligned}
\end{equation*}
and 
\begin{equation*}
\begin{aligned} \left\{\begin{array}{rcl} -\Delta_{\Phi_2} \widetilde{u}_n
&=& f_2(u_{n-1})| u_{n-1}|^{\alpha_2}_{L^{\Psi_2}} +g_2(u_{n-1})|
u_{n-1}|^{\gamma_2}_{L^{\Lambda_2}}\;\;\mbox{in} \;\;\Omega,\\ \vspace{.2cm}
\widetilde{u}_n&=&0\;\;\mbox{on}\;\;\partial\Omega. \end{array} \right.
\end{aligned}
\end{equation*}
where $\widetilde{u}_0 : = \overline{u}_2$ and $u_0 := \overline{u}_1.$
Arguing as in Lemma \ref{sub-supermethod} we obtain the result.

A sublinear system

In this section we use Lemma \ref{sub-supermethod-sys} and suitable
sub-supersolution pairs to prove the existence of solution for the the
nonlocal system 
\begin{equation*}
\left\{ 
\begin{array}{rcl}
-\Delta _{\Phi _{1}}u & = & v^{\beta _{1}}|v|_{\Psi _{1}}^{\alpha _{1}}\ %
\mbox{in}\ \Omega , \\ 
-\Delta _{\Phi _{2}}v & = & u^{\beta _{2}}|u|_{\Psi _{2}}^{\alpha _{2}}\ %
\mbox{in}\ \Omega , \\ 
u=v & = & 0\ \mbox{on}\ \partial \Omega ,%
\end{array}%
\right. \eqno{(P^{'}_S)}
\end{equation*}%
where $\alpha _{i}$ and $\beta _{i},i=1,2$ are constants saisfying certain
conditions. It is interesting to note that the set of hipothesis of the next
result is different from the system version of $(P)$ considered in \cite[%
Theorem 5.2]{CFL} in the constant exponent case.
\end{proof}

\begin{theorem}
\label{teo-sublinear-sys} Suppose that $\alpha_i, \beta_i \geq 0$ with $0 <
\alpha_1 + \beta_1 < l_i -1,0< \alpha_2 + \beta_2 < l_i -1, i =1,2$. Then $%
(P^{^{\prime }}_S)$ has a positive solution.
\end{theorem}

\begin{proof}
Let $\lambda>0$ and consider $z_{\lambda}\in W_{0}^{1,\Phi_1}(\Omega)\cap
L^{\infty}(\Omega)$ and $y_{\lambda}\in W_{0}^{1,\Phi_2}(\Omega)\cap
L^{\infty}(\Omega)$ the unique solutions of \eqref{probl-linear-lambda}
where $\lambda$ will be chosen later.

For $\lambda >0$ sufficiently large, by Lemma \ref{Tan-Fang} there is a
constant $K>0$ that does not depend on $\lambda$ such that 
\begin{equation}  \label{desig1-p-supsol-sis}
0<z_{\lambda}(x)\leq K\lambda^{\frac{1}{l_{1}-1}}\;\text{in}\;\Omega,
\end{equation}
and 
\begin{equation}  \label{ddesig1-p-supsol}
0<y_{\lambda}(x)\leq K\lambda^{\frac{1}{l_{2}-1}}\;\text{in}\;\Omega.
\end{equation}

Since $0< \alpha_1 + \beta_1 < l_2 -1,$ we can choose $\lambda >0$ large
enough satisfying $K^{\beta_1} |K|^{\alpha_1}_{L^{\Psi_1}}\lambda^{\frac{
\alpha_1 + \beta_1}{l_2 -1}} \leq \lambda.$ Thus from %
\eqref{desig1-p-supsol-sis} we have $y_{\lambda}^{\beta_1}
|y_{\lambda}|^{\alpha_1}_{L^{\Psi_1}} \leq \lambda \ \text{in} \ \Omega. $
Therefore 
\begin{equation*}
\begin{aligned} \left\{\begin{array}{rcl} -\Delta_{\Phi_1}z_{\lambda}&\geq&
y^{\beta_1}_{\lambda}
|y_{\lambda}|^{\alpha_1}_{L^{\Psi_1}}\;\;\mbox{in}\;\;\Omega,\\
\vspace{.2cm} z_{\lambda}&=&0\;\;\mbox{on}\;\;\partial\Omega. \end{array}
\right. \end{aligned}
\end{equation*}
From \eqref{ddesig1-p-supsol} and the fact that $0< \alpha_2 + \beta_2< l_1
-1$ we also have that 
\begin{equation*}
\begin{aligned} \left\{\begin{array}{rcl} -\Delta_{\Phi_2}y_{\lambda}&\geq&
z^{\beta_2}_{\lambda}
|z_{\lambda}|^{\alpha_2}_{L^{\Psi_2}}\;\;\mbox{in}\;\;\Omega,\\
\vspace{.2cm} y_{\lambda}&=&0\;\;\mbox{on}\;\;\partial\Omega, \end{array}
\right. \end{aligned}
\end{equation*}
for $\lambda >0 $ large enough.

Since $\partial \Omega$ is $C^2 ,$ there is a constant $\delta >0$ such that 
$d \in C^{2}(\overline{\Omega_{3 \delta}})$ and $|\nabla d(x)| \equiv 1,$
where $d(x):= dist(x,\partial \Omega)$ and $\overline{ \Omega_{3 \delta}}%
:=\{x \in \overline{\Omega}; d(x) \leq 3 \delta\}$. For $\sigma \in
(0,\delta)$ the function $\eta_i = \eta_i(k,\sigma),i=1,2$ defined by 
\begin{equation*}
\eta_i(x)=\left\{%
\begin{array}{lcl}
e^{kd(x)}-1 & \text{ if } & d(x)<\sigma, \\ 
e^{k\sigma}-1+\int_{\sigma}^{d(x)}ke^{k\sigma}\Big(\frac{2\delta-t}{%
2\delta-\sigma}\Big)^{\frac{m_i}{l_i-1}}dt & \text{ if } & \sigma\leq
d(x)<2\delta, \\ 
e^{k\sigma}-1+\int_{\sigma}^{2\delta}ke^{k\sigma}\Big(\frac{2\delta-t}{%
2\delta-\sigma}\Big)^{\frac{m_i}{l_i-1}}dt & \text{ if } & 2\delta \leq d(x)%
\end{array}
\right.
\end{equation*}
belongs to $C^{1}_{0}(\overline{\Omega})$ for $i=1,2,$ where $k>0$ is an
arbitrary constant. Note that 
\begin{equation*}
-\Delta_{\Phi}(\mu\eta_i)=%
\begin{cases}
- \mu k^2 e^{kd(x)} \frac{d}{dt} \left( \phi_i(t)t\right)\bigg\vert_{t= \mu
k e^{kd(x)}} - \phi_i(\mu k e^{kd(x)})\mu k e^{kd(x)}\Delta d \;\; \mbox{ if}%
\quad d(x)<\sigma, \\ 
\mu k e^{k \sigma}\left( \frac{m_i}{l_i-1}\right)\left(\frac{2\delta -d(x)}{%
2\delta- \sigma} \right)^{\frac{m_i}{l_i-1}-1}\left( \frac{1}{2\delta -
\sigma}\right) \frac{d}{dt}\left( \phi_i(t)t\right)\bigg\vert_{t = \mu k
e^{k \sigma} \left( \frac{2\delta -d(x)}{2\delta - \sigma}\right)} \\ 
- \phi_i\left( \mu k e^{k\sigma} \left( \frac{2\delta -d(x)}{2\delta - \sigma%
}\right)^{\frac{m_i}{l_i-1}}\right) \mu k e^{k\sigma} \left( \frac{%
2\delta-d(x)}{2\delta - \sigma}\right)^{\frac{m_i}{l_i-1}}\Delta d \;\; %
\mbox{ if}\quad \sigma < d(x)<2\delta, \\ 
0\;\; \mbox{ if}\quad 2\delta<d(x)%
\end{cases}%
\end{equation*}
for all $\mu >0$ and $i=1,2.$ Arguing as in \eqref{negative} we have $%
-\Delta_{\Phi_i} (\mu \eta_i) \leq 0, i=1,2$ for $k>0$ large enough when $0<
d(x) < \sigma$.

Reasoning as in \eqref{midle} we get 
\begin{equation}  \label{midle2}
- \Delta_{\Phi_1} (\mu \eta_1) \leq \max\left\{ \frac{K_1 }{2\delta-\sigma},
K_2 \right\} \max\{(\mu k e^{k\sigma})^{m_1-1}, (\mu k
e^{k\sigma})^{l_1-1}\},
\end{equation}
and 
\begin{equation}  \label{midle3}
- \Delta_{\Phi_2} (\mu \eta_2) \leq \max\left\{ \frac{K_3 }{2\delta-\sigma},
K_4 \right\} \max\{(\mu k e^{k\sigma})^{m_2-1}, (\mu k
e^{k\sigma})^{l_2-1}\},
\end{equation}
for $\sigma < d(x) < 2\delta,$ where $K_i , i=1,2,3,4$ are positive
constants that does not depend on $k>0.$

Consider $\sigma = \frac{\ln 2}{k}$ and $\mu = e^{-k}.$ We have $\eta_i(x)
\geq e^{k\sigma} -1 \geq 1$ for all $x \in \Omega$ and $i=1,2.$ Thus there
is a constant $K_5 >0$ such that 
\begin{equation*}
(\mu \eta_j)^{\beta_i} |\mu \eta_j|^{\alpha_i}_{L^{\Psi_i}} \geq
\mu^{\alpha_i + \beta_i} K_5, i,j=1,2, i \neq j
\end{equation*}
for $\sigma < d(x) < 2\delta.$

Since $0 < \alpha_i + \beta_i < l_i-1$, the L'Hospital's rule implies that 
\begin{equation*}
\lim_{k \rightarrow +\infty} \frac{k^{l_i-1}}{e^{k(l_i-1 - (\alpha_i +
\beta_i))}} = 0, i=1,2.
\end{equation*}
Thus, it is possible to consider $k_0>0$ large enough such that 
\begin{equation*}
K_5 \geq \max \left\{ K_1 \frac{1}{2\delta-\frac{\ln 2}{k}}, K_2
\right\}\max\{2^{m_1-1}, 2^{l_1-1}\} \frac{k^{l_1-1}}{e^{k(l_1-1 -(\alpha_1
+ \beta_1))}}
\end{equation*}
and 
\begin{equation*}
K_5 \geq \max \left\{ K_3 \frac{1}{2\delta-\frac{\ln 2}{k}}, K_4
\right\}\max\{2^{m_2-1}, 2^{l_2-1}\} \frac{k^{l_2-1}}{e^{k(l_2-1 -(\alpha_2
+ \beta_2))}},
\end{equation*}
for all $k \geq k_0.$ Thus for $k>0$ large enough we have $%
-\Delta_{\Phi_i}(\mu \eta_i) \leq (\mu \eta_j)^{\beta_i} |\mu
\eta_j|^{\alpha_i}_{L^{\Psi_i}},i,j=1,2,$ $i \neq j.$ for $\sigma < d(x) <
2\delta. $ If $d(x) > 2\delta$ we have $-\Delta_{\Phi_i} (\mu \eta_j) = 0
\leq (\mu \eta_j)^{\beta_i} |\mu \eta_i|^{\alpha_i}_{L^{\Psi_i}}, i,j=1,2$
with $i \neq j.$ For $k >0$ large enough we also have that $%
-\Delta_{\Phi_1}(\mu \eta_1) \leq - \Delta_{\Phi_1} z_\lambda ,
-\Delta_{\Phi_2} (\mu \eta_2) \leq - \Delta_{\Phi_2} y_\lambda$ in $\Omega$.
Therefore $\mu \eta_1 \leq z_{\lambda}, \mu \eta_2 \leq y_{\lambda}$ in $%
\Omega.$ The result follows.
\end{proof}

\subsection{A concave-convex system}

In this section we prove the existence of solution for a concave-convex
system of type 
\begin{equation*}
\left\{ 
\begin{array}{rcl}
-\Delta _{\Phi _{1}}u & = & \lambda v^{\beta _{1}}|v|_{\Psi _{1}}^{\alpha
_{1}}+\theta v^{\xi _{1}}|v|_{L^{\Lambda _{1}}}^{\gamma _{1}}\ \mbox{in}\
\Omega , \\ 
-\Delta _{\Phi _{2}}v & = & \lambda u^{\beta _{2}}|u|_{\Psi _{2}}^{\alpha
_{2}}+\theta u^{\xi _{2}}|u|_{L^{\Lambda _{2}}}^{\gamma _{2}}\ \mbox{in}\
\Omega , \\ 
u=v & = & 0\ \mbox{on}\ \partial \Omega ,%
\end{array}%
\right. \eqno{(P^{'})_{\lambda,\theta}}
\end{equation*}%
where $\alpha _{i},\beta _{i},\gamma _{i},\xi _{i},i=1,2$ are constants
satisfying certain conditions.

\begin{theorem}
Suppose that $\alpha_i,\beta_i, \gamma_i,\xi_i , i=1,2$ are nonnegative
constants and suppose that $0<\alpha_i+\beta_i<l_i-1,i=1,2$. The following
assertions hold

\vspace{0.2cm}

\noindent\textbf{$(i)$} If $m_2-1<\xi_1+\gamma_1$ and $m_1-1<\xi_2+\gamma_2
, $ then for each $\theta>0$ there exists $\lambda_{0}>0$ such that for each 
$\lambda\in(0,\lambda_{0})$ the problem $(P^{^{\prime }})_{\lambda,\theta}$
has a positive solution $u_{\lambda,\theta}.$

\vspace{0.2cm}

\noindent\textbf{$(ii)$} If $0< \alpha_1 + \beta_1 < l_2 -1, 0<\alpha_2 +
\beta_2 < l_1 -1 < \xi_1 + \gamma_1 < l_2 -1$ and $\xi_2 + \gamma_2 < l_1 -1$
then for each $\lambda >0$ there exists $\theta_0 >0$ such that for each $%
\theta \in (0,\theta_0)$ the problem $(P^{^{\prime }})_{\lambda,\theta}$ has
a positive solution $u_{\lambda,\theta}.$
\end{theorem}

\begin{proof}
Suppose that $(i)$ occurs. Consider $z_{\lambda }\in W_{0}^{1,\Phi _{1}}({%
\Omega })\cap L^{\infty }(\Omega )$ and $y_{\lambda }\in W_{0}^{1,\Phi _{2}}(%
{\Omega })\cap L^{\infty }(\Omega )$ the unique solutions of %
\eqref{probl-linear-lambda}, where $\lambda \in (0,1)$ will be chosen
before. Lemma \ref{Tan-Fang} imply that for $\lambda >0$ small enough there
exists a constant $K>0$ that does not depend on $\lambda $ such that 
\begin{equation}
0<z_{\lambda }(x)\leq K\lambda ^{\frac{1}{m_{1}-1}}\;\text{in}\;\Omega ,
\label{desig1-p-supsol-concavo-sys}
\end{equation}%
\begin{equation}
0<y_{\lambda }(x)\leq K\lambda ^{\frac{1}{m_{2}-1}}\;\text{in}\;\Omega .
\label{desig2-p-supsol-concavo-sys}
\end{equation}%
We will prove, for each $\theta >0,$ that there exists $\lambda _{0}>0$ such
that%
\begin{equation}
\lambda y_{\lambda }^{\beta _{1}}|y_{\lambda }|_{L^{\Psi _{1}}}^{\alpha
_{1}}+\theta y_{\lambda }^{\xi _{1}}|y_{\lambda }|_{L^{\Lambda
_{1}}}^{\gamma _{1}}\leq \lambda  \label{c_1}
\end{equation}%
and 
\begin{equation}
\lambda z_{\lambda }^{\beta _{2}}|z_{\lambda }|_{L^{\Psi _{2}}}^{\alpha
_{2}}+\theta z_{\lambda }^{\xi _{2}}|z_{\lambda }|_{L^{\Lambda
_{2}}}^{\gamma _{2}}\leq \lambda  \label{c_2}
\end{equation}%
in $\Omega .$ Since $0<\alpha _{i}+\beta _{i},i=1,2,$ $m_{2}-1<\xi
_{1}+\gamma _{1}$ and $m_{1}-1<\xi _{2}+\gamma _{2}$ there exists $\lambda
_{0}>0$ such that 
\begin{equation}
\lambda ^{\frac{m_{2}-1+\alpha _{1}+\beta _{1}}{m_{2}-1}}K^{\beta
_{1}}|K|_{L^{\Psi _{1}}}^{\alpha _{1}}+\theta \lambda ^{\frac{\xi
_{1}+\gamma _{1}}{m_{2}-1}}K^{\xi _{1}}|K|_{L^{\Lambda _{1}}}^{\gamma
_{1}}\leq \lambda  \label{c_3}
\end{equation}%
and 
\begin{equation}
\lambda ^{\frac{m_{1}-1+\alpha _{2}+\beta _{2}}{m_{1}-1}}K^{\beta
_{2}}|K|_{L^{\Psi _{2}}}^{\alpha _{2}}+\theta \lambda ^{\frac{\xi
_{2}+\gamma _{2}}{m_{1}-1}}K^{\xi _{2}}|K|_{L^{\Lambda _{2}}}^{\gamma
_{2}}\leq \lambda  \label{c_4}
\end{equation}%
for all $\lambda \in (0,\lambda _{0}).$ From %
\eqref{desig1-p-supsol-concavo-sys}, \eqref{desig2-p-supsol-concavo-sys}, %
\eqref{c_3} and \eqref{c_4} we obtain \eqref{c_1} and \eqref{c_2}. Therefore 
\begin{equation*}
-\Delta _{\Phi _{1}}z_{\lambda }\geq \lambda y_{\lambda }^{\beta
_{1}}|y_{\lambda }|_{L^{\Psi _{1}}}^{\alpha _{1}}+\theta y_{\lambda }^{\xi
_{1}}|y_{\lambda }|_{L^{\Lambda _{1}}}^{\gamma _{1}}
\end{equation*}%
and 
\begin{equation*}
-\Delta _{\Phi _{2}}y_{\lambda }\geq \lambda z_{\lambda }^{\beta
_{2}}|z_{\lambda }|_{L^{\Psi _{2}}}^{\alpha _{2}}+\theta z_{\lambda }^{\xi
_{2}}|z_{\lambda }|_{L^{\Lambda _{2}}}^{\gamma _{2}}
\end{equation*}%
in $\Omega $ for all $\lambda \in (0,\lambda _{0}).$

Consider $\eta_i, \delta, \sigma$ and $\mu$ as in the proof of Theorem \ref%
{teo-sublinear-sys}. Since $0< \alpha_i + \beta_i < l_i -1, i=1,2$ we have
that exists $\mu >0$ with $\mu \eta_1 \leq z_{\lambda},$ $\mu \eta_2 \leq
y_{\lambda}$ and the inequalities 
\begin{equation*}
-\Delta_{\Phi_1}(\mu \eta_1) \leq \lambda, -\Delta_{\Phi_1}(\mu \eta_1) \leq
\lambda (\mu \eta_2)^{\beta_1} |\mu \eta_2|^{\alpha_1}_{L^{\Psi_1}} + \theta
(\mu \eta_2)^{\xi_1}|\mu \eta_2|^{\gamma_1}_{L^{\Lambda_1}}
\end{equation*}
and 
\begin{equation*}
-\Delta_{\Phi_2}(\mu \eta_2) \leq \lambda, -\Delta_{\Phi_2}(\mu \eta_2) \leq
\lambda (\mu \eta_1)^{\beta_2} |\mu \eta_1|^{\alpha_2}_{L^{\Psi_2}} + \theta
(\mu \eta_1)^{\xi_2}|\mu \eta_1|^{\gamma_2}_{L^{\Lambda_2}}
\end{equation*}
in $\Omega.$ Thus by Lemma \ref{sub-supermethod-sys} we have the first part
of the result.

In order to prove the second part of the result consider $\eta_i,\delta$ and 
$\sigma_i,i=1,2$ as in the first part of the result and let $\lambda >0$
fixed. Since $0< \alpha_i + \beta_{i}< l_{i}-1,i=1,2$ there exists $\mu>0$
depending only on $\lambda $ such that 
\begin{equation*}
- \Delta_{\Phi_i}(\mu \eta_i) \leq 1 \ \text{and} \ -\Delta_{\Phi_i}(\mu
\eta_i) \leq \lambda (\mu \eta_j)^{\beta_i} |\mu
\eta_j|^{\alpha_i}_{L^{\Psi_i}}
\end{equation*}
in $\Omega $ with $i,j=1,2$ and $i \neq j.$

Let $M>0$ which will be chosen before and consider $z_{M}\in W_{0}^{1,\Phi
_{1}}(\Omega )\cap L^{\infty }(\Omega )$ and $y_{M}\in W_{0}^{1,\Phi
_{2}}(\Omega )\cap L^{\infty }(\Omega )$ solutions of%
\begin{equation*}
\begin{aligned} \left\{\begin{array}{rcl} -\Delta_{\Phi_1}z_{M}
&=&M\;\;\mbox{in} \;\;\Omega,\\ \vspace{.2cm}
z_M&=&0\;\;\mbox{on}\;\;\partial\Omega. \end{array} \right. \end{aligned}%
\hspace{2cm}\begin{aligned} \left\{\begin{array}{rcl} -\Delta_{\Phi_2}y_{M}
&=&M\;\;\mbox{in} \;\;\Omega,\\ \vspace{.2cm}
y_M&=&0\;\;\mbox{on}\;\;\partial\Omega. \end{array} \right. \end{aligned}
\end{equation*}

If $M >0$ is large enough, then by Lemma \ref{Tan-Fang} there exists a
constant $K>0$ that does not depend on $M$ such that 
\begin{equation}  \label{desig3-p-supsol-concavo}
0<z_{M}(x)\leq KM^{\frac{1}{l_{1}-1}}\;\text{in}\;\Omega,
\end{equation}
\begin{equation}  \label{desig4-p-supsol-concavo}
0<y_{M}(x)\leq KM^{\frac{1}{l_{2}-1}}\;\text{in}\;\Omega.
\end{equation}

In order to construct ${\overline{u}}_i ,{\overline{u}}_i,i=1,2$ we will
show that exist $\theta_0 >0$ depending on $\lambda$ with the following
property: if we consider $\theta \in (0,\theta_0)$ then there will be a
constant $M$ depending only on $\lambda$ and $\theta$ satisfying 
\begin{equation}  \label{c_5}
M \geq \lambda {y_M}^{\beta_1} |y_M|^{\alpha_1}_{L^{\Psi_1}} + \theta {y_M}%
^{\xi_1} |y_M|^{\gamma_1}_{L^{\Lambda_1}}
\end{equation}
and 
\begin{equation}  \label{c_6}
M \geq \lambda {z_M}^{\beta_2} |z_M|^{\alpha_2}_{L^{\Psi_2}} + \theta {z_M}%
^{\xi_2} |z_M|^{\gamma_2}_{L^{\Lambda_2}}
\end{equation}
in $\Omega.$ From \eqref{desig3-p-supsol-concavo} and %
\eqref{desig4-p-supsol-concavo} we have that \eqref{c_5} and \eqref{c_6}
occur if $M \geq 1$ and 
\begin{equation}  \label{equiv-f2}
\lambda \overline{K} M^{\rho-1} + \theta \overline{K}M^{\tau-1} \leq 1
\end{equation}
where $\overline{K}:= \max \{ K^{\beta_1} |K|^{\alpha_1}_{L^{\Psi_1}},
K^{\beta_2} |K|^{\alpha_2}_{L^{\Psi_2}}, K^{\xi_1}
|K|^{\gamma_1}_{L^{\Lambda_1}}, K^{\xi_2}|K|^{\gamma_2}_{L^{\Lambda_2}} \},$ 
\begin{equation*}
\rho:= \max\left\{\frac{\alpha_1 + \beta_1}{l_2 -1}, \frac{\alpha_2 + \beta_2%
}{l_1 -1}\right\} \ \text{and} \ \tau:= \max \left\{ \frac{\gamma_1 + \xi_1}{
l_2 -1}, \frac{\gamma_2 + \xi_2}{l_1 -1}\right\}.
\end{equation*}

Since $0<\rho <1$ and $\tau >1$ the function 
\begin{equation*}
\Psi (t)=\lambda \overline{K}t^{\rho -1}+\theta \overline{K}t^{\tau -1},t>0,
\end{equation*}%
belongs to $C^{1}\big((0,\infty ),\mathbb{R}\big)$ and attains a global
minimum at%
\begin{equation}
M_{\lambda ,\theta }:=M(\lambda ,\theta )=L\Biggl(\dfrac{\lambda }{\theta }%
\Biggl)^{\frac{1}{\tau -\rho }}  \label{minimum2}
\end{equation}%
where $L:=(\frac{1-\rho }{\tau -1})^{\frac{1}{\tau -\rho }}.$ The inequality %
\eqref{equiv-f2} is equivalent to find $M_{\lambda ,\theta }>0$ such that $%
\Psi (M_{\lambda ,\theta })\leq 1.$ By \eqref{minimum2} we have $\Psi
(M_{\lambda ,\theta })\leq 1$ if and only if 
\begin{equation*}
\lambda \overline{K}(1-\rho )^{\frac{\rho -1}{\tau -\rho }}\left( \frac{%
\lambda }{\theta }\right) ^{\frac{\rho -1}{\tau -\rho }}+\theta \overline{K}%
(1-\rho )^{\frac{\tau -1}{\tau -\rho }}\left( \frac{\lambda }{\theta }%
\right) ^{\frac{\tau -1}{\tau -\rho }}\leq 1
\end{equation*}%
Notice that the above inequality holds if $\theta >0$ is small enough
because $0<\rho <1$ and $\tau >1.$ Thus for $\lambda >0$ fixed there exists $%
\theta _{0}=\theta _{0}(\lambda )$ such that for each $\theta \in (0,\theta
_{0})$ there is a number $M=M_{\lambda ,\theta }>0$ such that %
\eqref{equiv-f2} occurs. Thus we can consider $M_{\lambda ,\theta }$ large
enough such that \eqref{c_5} and \eqref{c_6} occur. Therefore 
\begin{equation*}
-\Delta _{\Phi _{1}}z_{M}\geq \lambda y_{M}^{\beta _{1}}|y_{M}|_{L^{\Psi
_{1}}}^{\alpha _{1}}+\theta y_{M}^{\xi _{1}}|y_{M}|_{L^{\Lambda
_{1}}}^{\gamma _{1}}
\end{equation*}%
and 
\begin{equation*}
-\Delta _{\Phi _{2}}y_{M}\geq \lambda z_{M}^{\beta _{2}}|z_{M}|_{L^{\Psi
_{2}}}^{\alpha _{2}}+\theta z_{M}^{\xi _{2}}|z_{M}|_{L^{\Lambda
_{2}}}^{\gamma _{2}}
\end{equation*}%
Considering if necessary a smaller $\theta _{0}>0,$ we get 
\begin{equation*}
-\Delta _{\Phi _{1}}(\mu \eta _{1})\leq 1\leq M_{\lambda ,\theta _{0}}\leq
M_{\lambda ,\theta }
\end{equation*}%
and 
\begin{equation*}
-\Delta _{\Phi _{2}}(\mu \eta _{2})\leq 1\leq M_{\lambda ,\theta _{0}}\leq
M_{\lambda ,\theta }
\end{equation*}%
in $\Omega $ for all $\theta \in (0,\theta _{0})$ because $M_{\lambda
,\theta }\rightarrow +\infty $ as $\theta \rightarrow 0^{+}$ and $\theta
\longmapsto M_{\lambda ,\theta }$ is nonincreasing. Therefore $-\Delta
_{\Phi _{1}}(\mu \eta _{1})\leq -\Delta _{\Phi _{1}}z_{M},$ $-\Delta _{\Phi
_{2}}(\mu \eta _{2})\leq -\Delta _{\Phi _{2}}y_{M}$ in $\Omega .$ The weak
comparison principle implies that $\mu \eta _{1}\leq z_{M}$ and $\mu \eta
_{2}\leq y_{M}$ in $\Omega $. The proof is finished.
\end{proof}

\section{Final comments}

A slightly modification in the arguments of Lemma \ref{sub-supermethod-sys}
allow us to study a more general class of systems given by%
\begin{equation*}
\left\{ 
\begin{array}{rcl}
\label{problema-(P)}-\Delta _{\Phi _{1}}u & = & f_{1}(u,v)|v|_{L^{\Psi
_{1}}}^{\alpha _{1}}+g_{1}(u,v)|v|_{L^{\Lambda _{1}}}^{\gamma _{1}}\;\;%
\mbox{in}\;\;\Omega , \\ 
-\Delta _{\Phi _{2}}v & = & f_{2}(u,v)|u|_{L^{\Psi _{2}}}^{\alpha
_{2}}+g_{2}(u,v)|u|_{L^{\Lambda _{2}}}^{\gamma _{2}}\;\;\mbox{in}\;\;\Omega ,
\\ 
u=v & = & 0\;\;\mbox{on}\;\;\partial \Omega ,%
\end{array}%
\right. \eqno{(\widetilde{P})}
\end{equation*}%
with $f_{i},g_{i}:[0,+\infty )\times 0,+\infty )\rightarrow 0,+\infty
),i=1,2 $ nondecreasing continuous functions in the variables $u$ and $v.$
The arguments used in this work allow us to consider results in the case for
example when the functions $f_{i}$ and $g_{i}$ are power functions with
convenient exponents. In order to avoid of a more technical exposition we
choose to not prove results related with the case mentioned before, that is,
systems involving the variables $u$ and $v$ in the local the terms of each
equation of $(\widetilde{P}).$


\end{document}